\documentclass[12pt]{article}

\usepackage{cite}
\usepackage{amssymb}
\usepackage{amsthm}
\usepackage{amsmath}
\usepackage{amsfonts}
\usepackage{txfonts}
\usepackage{indentfirst}
\usepackage{enumerate}
\usepackage{hyperref}
\usepackage{color}

\usepackage[T1]{fontenc}
\usepackage[latin2]{inputenc}
\title{Sharp conditions to avoid collisions in singular Cucker-Smale interactions}

\author{
Jos\'e A. Carrillo\footnote{{\it Email address:} {\tt carrillo@imperial.ac.uk}}\\
{\it\footnotesize Department of Mathematics,}\\
{\it\footnotesize Imperial College London,}\\
{\it\footnotesize SW7 2AZ, London, United Kingdom}
\\
Young-Pil Choi\footnote{{\it Email address:} {\tt ychoi@ma.tum.de}}\\
{\it\footnotesize Fakult\"at f\"ur Mathematik,}\\
{\it\footnotesize Technische Universit\"at M\"unchen,}\\
{\it\footnotesize Boltzmannstra\ss e 3, 85748, Garching bei M\"unchen, Germany}
\\
Piotr B. Mucha\footnote{{\it Email address:} {\tt p.mucha@mimuw.edu.pl}}\\
{\it\footnotesize Institute of Applied Mathematics and Mechanics,}\\
{\it\footnotesize University of Warsaw,}\\
{\it\footnotesize ul. Banacha 2, 02-097 Warsaw, Poland}
\\
Jan Peszek\footnote{{\it Email address:} {\tt j.peszek@mimuw.edu.pl}}\\
{\it\footnotesize Institute of Applied Mathematics and Mechanics,}\\
{\it\footnotesize University of Warsaw,}\\
{\it\footnotesize ul. Banacha 2, 02-097 Warsaw, Poland}
}

\date{\today}

\renewcommand{\it}{\sl}

plus 6pt minus 12pt
\newcommand{\barint}{
         \rule[.036in]{.12in}{.009in}\kern-.16in
          \displaystyle\int  }

\begin{document}

\newtheorem{theo}{\bf Theorem}[section]
\newtheorem{coro}{\bf Corollary}[section]
\newtheorem{lem}{\bf Lemma}[section]
\newtheorem{rem}{\bf Remark}[section]
\newtheorem{defi}{\bf Definition}[section]
\newtheorem{ex}{\bf Example}[section]
\newtheorem{fact}{\bf Fact}[section]
\newtheorem{prop}{\bf Proposition}[section]
\newtheorem{prob}{\bf Problem}[section]

\makeatletter \@addtoreset{equation}{section}
\renewcommand{\theequation}{\thesection.\arabic{equation}}
\makeatother

\newcommand{\ds}{\displaystyle}
\newcommand{\ts}{\textstyle}
\newcommand{\ol}{\overline}
\newcommand{\wt}{\widetilde}
\newcommand{\ck}{{\cal K}}
\newcommand{\ve}{\varepsilon}
\newcommand{\vp}{\varphi}
\newcommand{\pa}{\partial}
\newcommand{\rp}{\mathbb{R}_+}
\newcommand{\hh}{\tilde{h}}
\newcommand{\HH}{\tilde{H}}
\newcommand{\cp}{{\rm cap}^+_M}
\newcommand{\hes}{\nabla^{(2)}}
\newcommand{\nn}{{\cal N}}
\newcommand{\dix}{\nabla_x\cdot}
\newcommand{\dv}{{\rm div}_v}
\newcommand{\di}{{\rm div}}
\newcommand{\pxi}{\partial_{x_i}}
\newcommand{\pmi}{\partial_{m_i}}
\newcommand{\tor}{\mathbb{T}}
\newcommand{\pot}{\mathcal{v}}

\newcommand{\R}{\mathbb{R}}
\newcommand{\lt}{\left}
\newcommand{\rt}{\right}
\newcommand{\om}{\Omega}
\newcommand{\mt}{\mathcal{T}}
\newcommand{\mc}{\mathcal{C}}
\newcommand{\ml}{\mathcal{L}}
\newcommand{\bq}{\begin{equation}}
\newcommand{\eq}{\end{equation}}


\maketitle

\begin{abstract}
We consider the Cucker-Smale flocking model with a singular communication weight $\psi(s) = s^{-\alpha}$ with $\alpha > 0$. We provide a critical value of the exponent $\alpha$ in the communication weight leading to global regularity of solutions or finite-time collision between particles. For $\alpha \geq 1$, we show that there is no collision between particles in finite time if they are placed in different positions initially. For $\alpha\geq 2$ we investigate a version of the Cucker-Smale model with expanded singularity i.e. with weight $\psi_\delta(s) = (s-\delta)^{-\alpha}$, $\delta\geq 0$. For such model we provide a uniform with respect to the number of particles estimate that controls the $\delta$-distance between particles. In case of $\delta = 0$ it reduces to the estimate of non-collisioness. 
\end{abstract}

%
%
%
%

\section{Introduction}
Mathematical description of dynamics of aggregating, swarming or flocking particles plays a significant role in modelling of various physical, biological and sociological phenomena.
 From the mathematical point of view such models take the form of transport-type PDE's or systems of ODE's, where the individuals are subjected to a force generated by their nonlocal interactions. Depending on the nature of the interactions, the individuals' behavior may differ for instance they can aggregate, align their velocities or disperse. There is a wide range of applications of such models that include such seemingly unrelated phenomena like distribution of goods, reaching a consensus among individuals or emergence of common languages in primitive cultures (see \cite{lang, cons, cons2, goods}). Among such models, our main subject of interest is the Cucker--Smale (in short, CS) flocking model introduced in \cite{cuc1} with the purpose of describing motion of self-propelled agents with tendency to flock. The CS dynamical system reads as follows:
\begin{align}
\begin{aligned}\label{cspart}
\left\{
\begin{array}{lcr}
\displaystyle \dot{x}_i=v_i, \quad i=1,\cdots, N,  \quad t > 0,\\[2mm]
\displaystyle \dot{v}_i=\frac{1}{N}\sum_{j=1}^N\psi(|x_i-x_j|)(v_j-v_i),
\end{array}
\right.
\end{aligned}
\end{align}
subject to the initial data
\bq\label{ini_cspart}
(x_i,v_i)(0) =: (x_{i0}, v_{i0}), \quad i =1,\cdots, N
\eq

Here $N$ is the number of particles, while $x_i(t)$ and $v_i(t)$ denote the position and velocity of $i$-th particle at the time $t$, respectively. The function $\psi$ is referred to as 
{\it the communication weight} and it is nonnegative and nonincreasing. The state of the art for CS model (and more generally, for kinetic models of interacting particles) is rich; it includes 
results in various directions, such as time asymptotics (see e.g. \cite{hakalaru, car}), pattern formation (see e.g. \cite{hajekaka, top}), models with additional deterministic (see e.g. \cite{car3, hahaki})
 or stochastic (see e.g. \cite{cuc4, aha2, choi, dua1, halele}) forces, aggregation with leaders (see e.g. \cite{cuc3, shen}), analysis of CS model with singular weight $\psi$ (see e.g. \cite{haliu, ahn1, jpe, jps}),
 passage from particle to the kinetic description for CS and similar models (see e.g. \cite{can, CH16, haliu, hatad, rec, deg1, deg3, mp}) and, particularly interesting from the point of view of this paper, collision avoidance 
(see e.g. \cite{cuc2, park}). We refer to \cite{CCP, CHL} for recent surveys.

In \cite{haliu}, the CS model with a singular communication weight of the form 
\begin{align}\label{psi}
\psi(s)=s^{-\alpha}\ \ \ {\rm for}\ \ \ \alpha>0
\end{align}
is considered and the time asymptotics behavior of solutions for the particle system \eqref{cspart} are provided. The question of existence, uniqueness and regularity of solutions was being consecutively answered after that work. In \cite{ahn1}, the particles' tendency to avoid collision is proved for $\alpha\geq 1$. To be more precise, the authors provide a set of initial configurations leading to no finite-time collision between particles and, as a consequence, the existence, uniqueness, and regularity of solutions to \eqref{cspart} are established for such initial data. On the other hand, existence, regularity, uniqueness of solutions to \eqref{cspart} in case of $\alpha\in(0,1)$ are investigated in \cite{jpe} and \cite{jps}. These results established a dichotomy: for $\alpha\in(0,1)$ the particles can collide and stick together and existence of solutions is obtained thanks to the low singularity of $\psi$ (particularly due to integrability of $\psi$ at $0$); on the other hand for $\alpha\in[1,\infty)$, the particles exhibit a tendency to avoid collisions.

The issue of collision-avoiding plays a significant role in quantitative analysis and applications of the model. Since system \eqref{cspart} with weight \eqref{psi} is singular only at times at which particles collide, knowledge that the particles do not collide enables us to immediately deduce existence, uniqueness and regularity of solutions. From the point of view of applications collision-avoiding is important in any situation when the described phenomena involves flocking of agents that naturally avoid collisions such as birds, fish or robots. In fact many models of flocking, such as in \cite{can}, that include effects of aggregation and alignment also include a repulsion effect to ensure the lack of collisions between particles.

The goal of this paper is to refine the results of \cite{ahn1} proving that for $\alpha\in[1,\infty)$ the CS particles indeed cannot collide (regardless of the initial configuration, as long as $x_{i0}\neq x_{j0}$ for $i\neq j$) and thus the solutions to \eqref{cspart} globally exist in time and are smooth and unique.

Further, in case of $\alpha\geq 2$ we improve our results by introducing a uniform with respect to the number of particles non-collisioness estimate. We prove the estimate for the following generalization of the CS particle system. Given $\delta \geq 0$ we introduce
\begin{align}
\begin{aligned}\label{dcspart}
\left\{
\begin{array}{lcr}
\displaystyle \dot{x}_i=v_i, \quad i=1,\cdots, N,  \quad t > 0,\\[2mm]
\displaystyle \dot{v}_i=\frac{1}{N}\sum_{j=1}^N\psi(|x_i-x_j|-\delta)(v_j-v_i),
\end{array}
\right.
\end{aligned}
\end{align}
subject to the initial data \eqref{ini_cspart} with $|x_{i0}-x_{j0}|>\delta$ for $i\neq j$.
In case of $\alpha>2$, the uniformly estimated quantity that determines $\delta$-distance between particles is given by
\begin{equation}\label{es}
\frac{1}{N^2}\sum_{i,j=1}^N(|x_i-x_j|-\delta)^{2-\alpha}\leq C \quad \mbox{for}\quad \alpha>2.
\end{equation}
Note that if $\delta = 0$ then \eqref{dcspart} reduces to \eqref{cspart} and \eqref{es} serves as a uniform estimate of non-collisioness. Similarly for $\delta>0$ it implies that 
the distances between particles are always greater than $\delta$ (provided that $|x_{i0}-x_{j0}|>\delta$ for $i\neq j$). This result can be viewed in the following way. We expand the set of singular points of $\psi$ from $\{0\}$ to $[0,\delta]$ and prove that the property that the solutions do not enter the singular set (in the sense that their distances do not enter this set) is preserved. It is also interesting from the point of view of applications, since it shows that the alignment kernel alone can be used to establish a minimal distance between particles (for which usually a separate repulsion kernel is utilized).

Another motivation of \eqref{es} comes from the fact that regularity of solutions to \eqref{cspart} can be controlled by the distances between particles. Thus estimates similar to \eqref{es} can be used in the passage from the particle system to the kinetic CS equation by mean-field limit as done for example in \cite{mp}. 

The paper is organized as follows. In Section \ref{noncoll} we present the proof of non-collisioness of solutions to \eqref{cspart} with $\alpha\geq 1$ from which we deduce the global existence, uniqueness and regularity of solutions to the particle system \eqref{cspart}. In Section \ref{estim}, we show the $N$-independent non-collisioness estimate for $\alpha\geq 2$.

%
%
%
%
\section{A global existence theory for the particle system}\label{noncoll}
In this section, we analyze the dynamics of the CS model with a singular weight given by \eqref{psi}. Our goal is to prove that for $\alpha\geq 1$ the particle system 
\eqref{cspart}-\eqref{ini_cspart} does not allow any collisions between particles and as a consequence admits a unique smooth solution globally in time. For this, we 
follow the ideas of \cite{haliu} establishing a system of locally dissipative differential inequalities (SDDI) for the quantities that control collisions between the particles.

%
%
%
%

From now on, $x=(x_1,\cdots,x_N)$, where $x_i=(x_{i,1},\cdots,x_{i,d})$ with $d \geq 1$ denotes the position of the particles, while $v = \dot{x}$ is their velocity. 
Throughout the paper $C$ denotes a generic positive constant that may change from line to line even in the same inequality. 

Let us be more precise about the meaning of collision between particles.

\begin{defi}
We say that the $i$-th and $j$-th particles collide at $t_0$ if and only if $x_i(t_0)=x_j(t_0)$ and we say that they stick together if they collide and $v_i(t_0)=v_j(t_0)$.
\end{defi}

The following proposition contains useful properties of solutions to the CS model showing the uniform bound estimate of the velocity in time. We refer to \cite{carchoha, jpe} for its proof.
\begin{prop}\label{bounds}
Let $(x,v)$ be a solution to (\ref{cspart}) on an arbitrary finite interval $[0,T]$, with singular communication weight given by (\ref{psi}). Then we have
\bq\label{vM}
\sup_{1 \leq i \leq N}|v_i(t)|\leq \sup_{1 \leq i \leq N}|v_{i0}| \quad \mbox{and} \quad \sup_{1 \leq i \leq N}|x_i(t)|\leq \sup_{1 \leq i \leq N}|x_{i0}| + 
\sup_{1 \leq i \leq N}|v_{i0}|\,t \nonumber \quad \mbox{for } t \geq 0.
\eq
\end{prop}
For notational simplicity, we set $M := \sup_{1 \leq i \leq N}|v_{i0}|$ and $R=R(T):= \sup_{1 \leq i \leq N}|x_{i0}|$ + TM.

%
%
%
%
In order to introduce a general strategy of the proofs, we begin with a simplified, one--dimensional case that serves as a good starting point and introduces the main ideas 
behind our argumentation in higher dimensions.
\begin{prop}\label{d1}
Let $d=1$ and $\alpha\geq 1$. Suppose that the initial data $(x_0,v_0)$ satisfy 
\[
x_{i0}\neq x_{j0} \quad \mbox{for } 1 \leq i \neq j \leq N.
\]
Then the system \eqref{cspart}-\eqref{ini_cspart} admits a unique smooth solution. Moreover, the trajectories of this solution do not collide in finite time, i.e.,
\[
x_i(t)\neq x_j(t) \quad \mbox{for } 1 \leq i \neq j \leq N, \quad t \geq 0.
\]
\end{prop}
\begin{proof}
Suppose  $\alpha\geq 1$ and that we have an arbitrary non--collision initial data $(x_0,v_0)$. Then since $\psi(|x_i-x_j|)$ is regular as long as $x_i\neq x_j$, 
then there exists a unique smooth solution to \eqref{cspart} on the time interval $[0,t_0)$, where $t_0$ is a supposed time of the first collision between any number of particles. Suppose that $l$th particle collides with some other particles at $t_0$ and denote by $[l]$ the set of all indices $i$ such that $i$th particle collides with $l$th particle at $t_0$. 
We will show that $t_0 < \infty$ actually does not exist, i.e., $t_0 = \infty$. Since $x(t)$ is Lipschitz continuous on $[0,t_0)$, it can be continuously extended to $t_0$, which means that $|x_i(t)-x_j(t)|\to 0$ as $t\to t_0$ if $i,j\in[l]$. Suppose without a loss of generality that $x_i(t)$ with $i\in[l]$ is the furthest to the right particle  
from $[l]$, i.e., we have $x_j(t) < x_i(t)$ for all $j\in[l], j\neq i$ on $[0,t_0)$. For $t\in[0,t_0)$, let
\begin{align*}
r(t):= \sum_{j\in[l]}(v_i(t)-v_j(t)).
\end{align*}
Then it follows from $(\ref{cspart})_2$ that
\begin{align}\label{d1eq1}
\begin{aligned}
\frac{d}{dt}r 
&= \frac{1}{N}\sum_{j\in[l]}\sum_{k=1}^N\psi(|x_k-x_i|)(v_k-v_i)
-\frac{1}{N}\sum_{j\in[l]}\sum_{k=1}^N\psi(|x_k-x_j|)(v_k-v_j)\\
&= \frac{1}{N}\sum_{j,k\in[l]}\psi(|x_k-x_i|)(v_k-v_i)
-\frac{1}{N}\sum_{j,k\in[l]}\psi(|x_k-x_j|)(v_k-v_j)\\
&\quad +\frac{1}{N}\sum_{\stackrel{j\in[l]}{k\notin[l]}} \left(\psi(|x_k-x_i|)(v_k-v_i)-\psi(|x_k-x_j|)(v_k-v_j)\right)\\
&=: I_1 + I_2 + I_3,
\end{aligned}
\end{align} 
where by substituting indices $j$ and $k$ from the set $[l]$, we obtain
\begin{align*}
I_2 = \frac{1}{N}\sum_{j,k\in[l]}\psi(|x_k-x_j|)(v_j-v_k) = -I_2=0
\end{align*}
and $I_3$ is integrable on $[0,t_0)$. On the other hand, since $x_i(t)>x_k(t)$ for $k\in[l]$, we find
\begin{align*}
\sum_{k\in[l]}\left(\Psi(|x_k-x_i|)\right)' = \sum_{k\in[l]}\left(\Psi(x_i-x_k)\right)' = \sum_{k\in[l]}\psi(x_i-x_k)(v_i-v_k) = I_1.
\end{align*}
Here $\Psi$ is the primitive of $\psi$, i.e.,
\bq\label{def_pri}
\Psi(s) = \left\{ \begin{array}{ll}
 \ln(s) & \textrm{if $\alpha = 1$,}\\[1mm]
\displaystyle \frac{1}{1-\alpha}s^{1-\alpha} & \textrm{if $\alpha > 1$}.
  \end{array} \right.
\eq
Integrating \eqref{d1eq1} over $[0,t]$ with $t < t_0$ it yields
\[
r(t)-r(0)  = \sum_{k\in[l]}\Psi(|x_{0,k}-x_{0,i}|)-\sum_{k\in[l]}\Psi(|x_k(t)-x_i(t)|) + \int_0^t I_3\,ds,
\]
and by the bounds from Proposition \ref{bounds}, this implies that the function
\[
t\mapsto \sum_{k\in[l]}\Psi(|x_k(t)-x_i(t)|)
\]
is bounded on the time interval $[0,t_0)$. However, this cannot be the case since $|x_k(t)-x_i(t)|\to 0$ as $t\to t_0$ for all $k\in[l]$ and $\Psi$ has a singularity at $0$. 
Thus this contradicts the supposed existence of a finite time of collision. Hence the unique smooth solution that was previously assumed to exist on the time interval $[0,t_0)$
 can be actually prolonged up to an arbitrary finite time $T$. This completes the proof.
\end{proof}

We are now in a position to provide the proof of non-collisioness for the multi-dimensional case. 

\begin{theo}\label{thm_main1} Let $d \geq 1$ and $\alpha \geq 1$. Suppose that the initial data $(x_0,v_0)$ are non-collisional, i.e. they satisfy 
\[
x_{i0}\neq x_{j0} \quad \mbox{for } 1 \leq i \neq j \leq N.
\]
Then the system \eqref{cspart}-\eqref{ini_cspart} admits a unique smooth solution. Moreover, the trajectories of this solution are also non-collisional, i.e.,
\[
x_i(t)\neq x_j(t) \quad \mbox{for } 1 \leq i \neq j \leq N, \quad t \geq 0.
\]

\end{theo}

\begin{proof}[Proof of Theorem \ref{thm_main1}]
At the very beginning of the proof, let us reformulate the thesis in a more suitable way. Let us fix $\alpha\geq 1$ and $T>0$. We need to prove that on $[0,T]$ 
there exists a unique solution to (\ref{cspart}) and that collisions of particles are impossible. However, since at $t=0$ the particles have distinct 
positions and the communication weight $\psi$ is singular only in a neighborhood of $0$, local existence of a unique smooth solution is standard. In fact, there 
are two possibilities: the particles do not collide in $[0,T]$ and the local existence can be extended up to $[0,T]$ or there exists $t_0\in(0,T]$, the first time of
 collision of any particles and we only know that the solution exists and is unique and smooth on $[0,t_0)$. Thus let us suppose that such $t_0$ exists. Then, by its definition,
 there exists an index $l=1,...,N$ such that the $l$th particle collides with some other particles. Let us denote by $[l]$ the set of all those indices $j \in \{1,\cdots, N\}$ that 
the $j$th particle collides with $l$th particle, i.e.,
\begin{align}
&|x_l(t)-x_j(t)|\to 0\ \ \ {\rm as}\ \ \  t\to t_0\ \ \ {\rm for\ all}\ \ \ j\in[l],\nonumber\\
&|x_l(t)-x_j(t)|\geq\delta>0\ \ \ {\rm in}\ \ \ [0,t_0)\ \ \ {\rm for\ all}\ \ \ j\notin[l]\ \ \ {\rm and\ some}\
 \ \ \delta>0.\label{delta}
\end{align}
Thus by the definition of $t_0$ there exists at least one set $[l]$ such that $\lt|[l]\rt| > 1$. Let us fix one of such sets denoting it simply by $[l]$. Let us also  denote
\begin{align*}
\|x\|_{[l]}(t):= \sqrt{\sum_{i,j\in[l]}|x_i(t)-x_j(t)|^2} \quad \mbox{and} \quad \|v\|_{[l]}(t):= \sqrt{\sum_{i,j\in[l]}|v_i(t)-v_j(t)|^2},
\end{align*}
where the sum is taken over all $i,j\in[l]$. Then we get
\begin{align}\label{contr}
\|x\|_{[l]}(t)\to 0\ \ \  {\rm as}\ \ \  t\nearrow t_0.
\end{align}
In particular if we show that it is impossible for $\|x\|_{[l]}$ to converge to $0$ as $t\nearrow t_0$, then it will contradict the assumption that $t_0$ is the first time of collision of any particles. We have
\begin{align*}
\frac{d}{dt}\|x\|_{[l]}^2 = 2\sum_{i,j\in[l]}(x_i-x_j)\cdot(v_i-v_j)\leq 2\sqrt{\sum_{i,j\in[l]}|x_i-x_j|^2}\sqrt{\sum_{i,j\in[l]}|v_i-v_j|^2}= 2\|x\|_{[l]}\|v\|_{[l]},
\end{align*}
which implies that
\begin{align}\label{rprime}
\left|\frac{d}{dt}\|x\|_{[l]}\right|\leq \|v\|_{[l]}.
\end{align}
On the other hand, we find from $(\ref{cspart})_2$ that
\begin{align}\label{dv}
\frac{d}{dt}\|v\|_{[l]}^2 &= 2\sum_{i,j\in[l]}(v_i-v_j)\cdot\left[\frac{1}{N} \sum_{k=1}^N\psi(|x_k-x_i|)(v_k-v_i)- \frac{1}{N}\sum_{k=1}^N\psi(|x_k-x_j|)(v_k-v_j)\right]\\
&=\frac{2}{N}\left(\sum_{i,j,k\in[l]}+\sum_{\stackrel{i,j\in[l]}{k\notin[l]}}\right) \left[\psi(|x_i-x_k|)(v_i-v_j)\cdot(v_k-v_i)-\psi(|x_j-x_k|)(v_i-v_j)\cdot(v_k-v_j)\right]\nonumber\\
&=: J_1 + J_2.\nonumber
\end{align}
The estimate of $J_1$ follows by its antisymmetry, which we explain below.
By substituting indices $k$ and $i$ in the first term of $J_1$, we obtain
\begin{align}\label{symm}
&\frac{2}{N}\sum_{i,j,k\in[l]}\psi(|x_k-x_i|)(v_i-v_j)\cdot(v_k-v_i) \\
&\quad =\frac{1}{N}\sum_{i,j,k\in[l]}\psi(|x_k-x_i|)(v_i-v_j)\cdot(v_k-v_i) +\frac{1}{N}\sum_{i,j,k\in[l]}\psi(|x_k-x_i|)(v_k-v_j)\cdot(v_i-v_k)\nonumber\\
&\quad = \frac{1}{N}\sum_{i,j,k\in[l]}\psi(|x_k-x_i|)(v_i-v_k)\cdot(v_k-v_i)\nonumber\\
&\quad =-\frac{|[l]|}{N}\sum_{i,j\in[l]}\psi(|x_i-x_j|)|v_i-v_j|^2,\nonumber
\end{align}
and similarly
\begin{align*}
- \frac{2}{N}\sum_{i,j,k\in[l]}^N\psi(|x_k-x_j|)(v_i-v_j)\cdot(v_k-v_j) = -\frac{|[l]|}{N}\sum_{i,j\in[l]}\psi(|x_i-x_j|)|v_i-v_j|^2.
\end{align*}
This yields
\begin{align*}
J_1 \leq -\frac{2|[l]|}{N}\sum_{i,j\in[l]}\psi(|x_i-x_j|)|v_i-v_j|^2.
\end{align*}
Since we have $|x_i-x_j|\leq  \|x\|_{[l]}$ for all $i,j\in[l]$, by monotonicity of $\psi$ we obtain
\begin{align*}
J_1 \leq-2c_0\psi(\|x\|_{[l]})\sum_{i,j\in[l]}|v_i-v_j|^2 = -2c_0\psi(\|x\|_{[l]})\|v\|_{[l]}^2,
\end{align*}
where $c_0 := |[l]|/N$. The estimate of $J_2$ follows by the fact that $|x_k-x_i|$ is separated from $0$ for $i\in[l]$ and $k\notin[l]$, which is written explicitly in (\ref{delta}). We have
\begin{align*}
J_2 &= \frac{2}{N}\sum_{\stackrel{i,j\in[l]}{k\notin[l]}}\psi(|x_i-x_k|)(v_i-v_j)\cdot(v_j-v_i) + \frac{2}{N}\sum_{\stackrel{i,j\in[l]}{k\notin[l]}}\left(\psi(|x_i-x_k|)-\psi(|x_j-x_k|)\right)(v_i-v_j)\cdot(v_k-v_j) \\
&= \underbrace{-\frac{2}{N}\sum_{\stackrel{i,j\in[l]}{k\notin[l]}} \psi(|x_i-x_k|)|v_i-v_j|^2}_{\leq 0} + \frac{2}{N}\sum_{\stackrel{i,j\in[l]}{k\notin[l]}}\left(\psi(|x_i-x_k|)-\psi(|x_j-x_k|)\right)(v_i-v_j)\cdot(v_k-v_j) \\
&\leq \frac{2L(\delta)}{N}\sum_{\stackrel{i,j\in[l]}{k\notin[l]}}|(v_i-v_j)\cdot(v_k-v_j)||x_i-x_j|,
\end{align*}
where $L(\delta)$ is the Lipschitz constant of $\psi$ in the interval $(\delta,\infty)$. Moreover by Proposition \ref{bounds} velocity $|v_k-v_j|$ is bounded by $2M$ and ultimately by H\"older's inequality
\[
J_2 \leq \frac{4ML(\delta)}{N}\sum_{\stackrel{i,j\in[l]}{k\notin[l]}}|v_i - v_j||x_i - x_j| = \frac{4ML(\delta)(N - |[l]|)}{N}\sum_{i,j \in [l]}|v_i - v_j||x_i - x_j| \leq 2c_1 \|v\|_{[l]}\|x\|_{[l]},
\]
where $c_1$ is a positive constant given by
\[
c_1 := \frac{2ML(\delta)(N- |[l]|)}{N}.
\]
Combining all of the above estimates, we have
\[
\frac{d}{dt}\|v\|_{[l]}^2\leq -2c_0\psi\lt(\|x\|_{[l]})\rt)\|v\|_{[l]}^2 + 2c_1\|v\|_{[l]}\|x\|_{[l]}.
\]
Note that if $\|v\|_{[l]} \neq 0$, then we get $\frac{d}{dt}\|v\|_{[l]}^2 = 2\|v\|_{[l]}\frac{d}{dt}\|v\|_{[l]}$. On the other hand, if $\|v\|_{[l]} \equiv 0$ on an open sub interval of $(s,t_0)$, then it is clear $\frac{d}{dt}\|v\|_{[l]} \equiv 0 \leq -c_0\psi\lt(\|x\|_{[l]})\rt)\|v\|_{[l]} + c_1\|x\|_{[l]} = c_1\|x\|_{[l]}$ on that interval. Thus we obtain
\[
\frac{d}{dt}\|v\|_{[l]}\leq -c_0\psi\lt(\|x\|_{[l]})\rt)\|v\|_{[l]} + c_1\|x\|_{[l]} \quad a.e. \mbox{ on } (s,t_0).
\]
Applying Gronwall's inequality to the above differential inequality together with the continuity of $\|v\|_{[l]}$ on the time interval $(s,t_0)$ yields
\begin{align}\label{rbis}
\|v\|_{[l]}(t)\leq \left(c_1\int_s^t\|x\|_{[l]}(\tau) e^{c_0\int_s^\tau\psi(\|x\|_{[l]}(\sigma))d\sigma}d\tau +\|v\|_{[l]}(s)\right)\exp\lt(-c_0\int_s^t\psi(\|x\|_{[l]}(\tau))\,d\tau\rt).
\end{align}

The estimates \eqref{rprime} and \eqref{rbis} enable us to finally approach the conclusion of the proof. Let us recall $\Psi$ is the primitive of $\psi$ given in \eqref{def_pri}. 
Then, by (\ref{rprime}) in the interval $(s,t_0)$, we have
\begin{align*}
\left|\Psi(\|x\|_{[l]}(t))\right| &= \left|\int_s^t\frac{d}{dt}\Psi(\|x\|_{[l]}(\tau))\,d\tau
 +\Psi(\|x\|_{[l]}(s))\right|\\
&=\left|\int_s^t\psi(\|x\|_{[l]}(\tau))\left(\frac{d}{dt}\|x\|_{[l]}\right)(\tau)\,d\tau
+ \Psi(\|x\|_{[l]}(s))\right|\\
&\leq \int_s^t\psi(\|x\|_{[l]}(\tau))\|v\|_{[l]}(\tau)\,d\tau +|\Psi(\|x\|_{[l]}(s))|.
\end{align*}
We apply (\ref{rbis}) to obtain
\begin{align}
\left|\Psi(\|x\|_{[l]}(t))\right|
&\leq \!\!\int_s^t\!\!\psi(\|x\|_{[l]}(\tau))\left(c_1\int_s^\tau\|x\|_{[l]}(\sigma) e^{c_0\int_s^\sigma\psi(\|x\|_{[l]}(\rho))d\rho}d\sigma +\|v\|_{[l]}(s)\right) e^{-c_0\int_s^\tau\psi(\|x\|_{[l]}(\sigma))d\sigma}d\tau\nonumber\\
&\quad +|\Psi(\|x\|_{[l]}(s))| \cr
&=: A + |\Psi(\|x\|_{[l]}(s))|\label{a}.
\end{align}
In order to estimate $A$ in the above inequality let us simplify the notation by taking $a:=\psi(\|x\|_{[l]})$. Then we have
\begin{align*}
A &= c_1\int_s^ta(\tau)\int_s^\tau\|x\|_{[l]}(\sigma) e^{c_0\int_s^\sigma a(\rho)d\rho}d\sigma\ e^{-c_0\int_s^\tau a(\sigma)d\sigma}\,d\tau\\
&\quad + \|v\|_{[l]}(s)\int_s^ta(\tau)e^{-c_0\int_s^\tau a(\sigma)d\sigma}\,d\tau =: A_1 + A_2.
\end{align*}
By Proposition \ref{bounds} there exists a constant $c_3=c_3(T) = \max\{M,R\}$ such that
\[
c_1\|x\|_{[l]}\leq c_3\quad \mbox{and}\quad \|v\|_{[l]}\leq c_3 \quad\mbox{on}\quad [0,t_0).
\]
Therefore we have
\begin{align*}
A_1\leq c_3\int_s^t\left(\int_s^\tau e^{c_0\int_s^\sigma a(\rho)d\rho}d\sigma\right)\left(a(\tau) e^{-c_0\int_s^\tau a(\sigma)d\sigma}\right)\,d\tau
\end{align*}
and noting that
\begin{align*}
\int_s^t a(\tau)e^{-c_0\int_s^\tau a(\sigma)\,d\sigma}d\tau = -\frac{1}{c_0}e^{-c_0\int_s^t a(\tau)\,d\tau}
\end{align*}
by integration by parts we obtain
\begin{align*}
A_1 &\leq -\frac{c_3}{c_0}\int_s^te^{c_0\int_s^\tau a(\sigma)d\sigma}d\tau \ e^{-c_0\int_s^ta(\tau)d\tau} + \frac{c_3}{c_0}\int_s^t e^{c_0\int_s^\tau a(\sigma)d\sigma} e^{-c_0\int_s^\tau a(\sigma)\,d\sigma}\\
&\leq \frac{c_3}{c_0}T.
\end{align*}
Similar calculation reveals that
\begin{align*}
A_2\leq \frac{c_3}{c_0}\left(1-e^{-c_0\int_s^t a(\tau)d\tau}\right)\leq \frac{c_3}{c_0}
\end{align*}
and altogether
\begin{align*}
A\leq \frac{c_3}{c_0}(T+1).
\end{align*}
We apply the above estimate of $A$ in \eqref{a} to obtain
\begin{align*}
\left|\Psi(\|x\|_{[l]}(t))\right| \leq \frac{c_3}{c_0}(T+1) +|\Psi(\|x\|_{[l]}(s))|
\end{align*}
and by Proposition \ref{bounds} the right-hand side of the above inequality is bounded regardless of the choice of $s\in[0,t_0)$. Thus $|\Psi(\|x\|_{[l]})|$ is also bounded in $(s,t_0)$,
 which subsequently implies that $\|x\|_{[l]}$ is separated from $0$ in $(s,t_0)$. To be more precise, we have
\begin{align*}
\|x\|_{[l]}(t)\geq \left|\Psi^{-1}\left(\frac{c_3}{c_0}(T+1) + |\Psi(\|x\|_{[l]}(s))|\right)\right|>0.
\end{align*}
This contradicts \eqref{contr}, and consequently it proves that the assumption of existence of $t_0$, the first time of collision between any number of the particles was false.
\end{proof}

\begin{rem} The estimates of derivatives of $\|x\|_{[l]}$ and $\|v\|_{[l]}$, in principle, are similar to (and to some degree were inspired by) the SDDI derived in \cite{haliu}. In fact, we 
divided the particles into two groups: one of those particles that collide with each other at $t_0$ and the second group of those that do not collide with the first group. Next, for the first group 
we basically derived the SDDI similarly to \cite{haliu} with some additional terms that originated from the interaction between the groups.
\end{rem}

\begin{rem}As a direct application of \cite[Theorem 3.1]{ahn1}, we also have the flocking estimates: Consider the CS particle system \eqref{cspart}-\eqref{ini_cspart}. Suppose that the initial
 configurations $(x_0,v_0)$ satisfy
\[
x_{i0}\neq x_{j0} \quad \mbox{for } 1 \leq i \neq j \leq N \quad \mbox{and} \quad \|v_0 - v_c(0)\|_{\ell^\infty} < \frac12 \int_{2\|x_0 - x_c(0)\|_{\ell^\infty}}^\infty \psi(s)\,ds,
\]
where $x_c(t)$ and $v_c(t)$ denote the average quantities of the position and velocity, i.e.,
\[
x_c(t) = \frac1N\sum_{i=1}^N x_i(t) \quad \mbox{and} \quad v_c(t) = \frac1N\sum_{i=1}^N v_i(t),
\]
respectively. Then there exist positive constants $c_1 > c_0(t) > 0$ such that
\[
c_0(t)\leq |x_i(t) - x_j(t)| \leq c_1 \,, i\neq j\quad \mbox{and} \quad \|v(t) - v_c(0)\|_{\ell^\infty} \leq \|v_0 - v_c(0)\|_{\ell^\infty} e^{-\psi(2c_1) t},
\]
for $t \geq 0.$ Note that $c_0(t)$ might go to zero as $t\to\infty$ but $c_1$ is uniform in time. This implies that the communication rate is bounded below uniformly in time giving the exponential convergence to the mean velocity, see \cite{ahn1,car} for details. In particular, for the critical case, $\alpha = 1$, if we assume 
\[
x_{i0}\neq x_{j0} \quad \mbox{for } 1 \leq i \neq j \leq N,
\]
then there is no collision for all time and we have exponential flocking estimate because $\int^\infty \psi(s)\,ds = \infty$.
\end{rem}

%
%
%
%
\section{Uniform estimate of minimal interparticle distance}\label{estim}

In this part, we consider system \eqref{dcspart} denoting the singular communication weights $\psi_\delta(s) = \psi(s-\delta)$, where $\delta \geq 0$ is a control parameter to 
make particles stay away from each other. We subject system \eqref{dcspart} to the initial data
\bq\label{dini_cspart}
(x_i,v_i)(0) =: (x_{i0}, v_{i0}), \quad i =1,\cdots, N.
\eq
It is worth mentioning that the strategy used in Section \ref{noncoll} cannot be used for the system \eqref{dcspart}. Thus we propose a new argument based on an energy estimate. 
For $\delta=0$, Theorem \ref{thm_d} implies the lack of collisions between particles and on top of that, it provides a uniform with respect to $N$ estimate of non-collisioness for system \eqref{cspart}.

For a fixed $\delta>0$ we introduce a function $\ml^\beta(t)$ which determines the $\delta$-distance between particles:
\[
\ml^\beta(t) := \frac{1}{N^2}\sum_{i,j=1}^N\left(|x_i(t) - x_j(t)|-\delta\right)^{-\beta} \quad \mbox{with} \quad \beta > 0.
\]
Note that there exists a $\beta > 0$ such that $\ml^\beta(t) < \infty$ for $t \in [0,T]$ if and only if the distances between particles are no less than $\delta$ for $t \in [0,T]$. Thus the 
solution to the $\delta$-CS system \eqref{dcspart} is well-defined provided that we start from admissible initial data i.e., $|x_{i0}-x_{j0}|>\delta$ for $i\neq j$. Notice that the definition of $\ml^\beta$ and its property of  determining the $\delta$-distance between particles is valid for all $\beta>0$. However in the proof of Theorem \ref{thm_d} we take $\beta = \alpha -2$. 

We introduce the maximal life-span $T(x_0)$ of the initial datum $x_0$:
\[
T(x_0):= \sup\lt\{ s \in \R_+: \exists \mbox{ solution $(x(t),v(t))$ for the system \eqref{dcspart} in a time-interval $[0,s)$}  \rt\}.
\]
\begin{theo}\label{thm_d}Suppose that $\alpha \in [2,\infty)$ and the initial data $x_0$ satisfy
\begin{equation}\label{assume_basicd}
|x_{i0} - x_{j0}|>\delta \quad \mbox{for any} \quad 1 \leq i \neq j \leq N.
\end{equation}
Then there exists a global smooth solution $(x(t),v(t))$ to the system \eqref{dcspart}, i.e., $T(x_0) = \infty$. Moreover for $t\geq 0$ we have
\begin{align}
&\lt|\frac{1}{N^2}\sum_{i,j = 1}^N\log \lt(|x_i(t) - x_j(t)|-\delta\rt)\rt| \cr
&\qquad \quad \leq \lt|\frac{1}{N^2}\sum_{i,j = 1}^N\log \lt(|x_i(0) - x_j(0)|-\delta\rt)\rt| + \frac{T_0}{2} + \frac{1}{2N}\sum_{i=1}^N |v_i(0)|^2 \quad \mbox{if} \quad \alpha = 2.\label{estim1d}\\
&\ml^{\alpha - 2}(t) \leq \ml^{\alpha - 2}(0)e^{Ct} + Ce^{Ct}\frac{1}{N}\sum_{i=1}^N |v_i(0)|^2 \hspace{4.3cm} \mbox{if} \quad \alpha>2,\label{estim2d}
\end{align}
where $C$ is a constant depending only on $\alpha$.
\end{theo}
\begin{proof}
Throughout the proof we will denote for simplicity $T_0:=T(x_0)$. 
Clearly, if the distances between particles are bigger than $\delta$ for any finite time then the solution can be prolonged indefinitely and $T_0=\infty$. Moreover if \eqref{estim1d} 
or \eqref{estim2d} holds then $\inf_{t\geq 0}\min_{i,j=1,...,N}|x_i(t)-x_j(t)|>\delta$ and thus also $T_0=\infty$. Therefore in order to finish the proof it suffices to show \eqref{estim1d} and \eqref{estim2d}.

First we provide the energy dissipation of the system \eqref{dcspart}. By $\eqref{dcspart}_2$ and a symmetry argument, we have 
\[
\frac{d}{dt} \frac{1}{N}\sum_{i=1}^N |v_i(t)|^2 + \frac{1}{N^2}\sum_{i,j=1}^N\frac{|v_i(t) - v_j(t)|^2}{(|x_i(t) - x_j(t)|-\delta)^\alpha} = 0 \quad \mbox{for} \quad t \in [0,T_0),
\]
and this yields
\begin{equation}\label{est_energy}
\frac{1}{N}\sum_{i=1}^N |v_i(t)|^2 + \int_0^t \frac{1}{N^2}\sum_{i,j=1}^N\frac{|v_i(s) - v_j(s)|^2}{(|x_i(s) - x_j(s)|-\delta)^\alpha} ds =  \frac{1}{N}\sum_{i=1}^N|v_i(0)|^2 \quad \mbox{for} \quad t \in [0,T_0).
\end{equation}
We divide the proof into two cases: $\alpha = 2$ and $\alpha > 2$.  (i) $\alpha = 2$: We estimate
$$\begin{aligned}
\lt|\frac{d}{dt} \frac{1}{N^2}\sum_{i,j = 1}^N\log \lt(|x_i(t) - x_j(t)|-\delta\rt)\rt| &= \lt|\frac{1}{N^2}\sum_{i,j=1}^N \frac{(x_i(t) - x_j(t))\cdot (v_i(t) - v_j(t))}{|x_i(t) - x_j(t)|}\frac{1}{(|x_i(t)-x_j(t)|-\delta)}\rt|\cr
&\leq \frac{1}{N^2}\sum_{i,j=1}^N \frac{|v_i(t) - v_j(t)|}{|x_i(t) - x_j(t)|-\delta}\cr
&\leq \frac{1}{2} + \frac{1}{2N^2}\sum_{i,j=1}^N \frac{|v_i(t) - v_j(t)|^2}{(|x_i(t) - x_j(t)|-\delta)^2},
\end{aligned}$$
for $t \in [0,T_0)$. We use the energy estimate \eqref{est_energy} to have
$$\begin{aligned}
&\lt|\frac{1}{N^2}\sum_{i,j = 1}^N\log \lt(|x_i(t) - x_j(t)|-\delta\rt)\rt| \cr
&\quad \leq \lt|\frac{1}{N^2}\sum_{i,j = 1}^N\log \lt(|x_i(0) - x_j(0)|-\delta\rt)\rt| + \frac{t}{2} + \frac{1}{2}\int_0^t \frac{1}{N^2}\sum_{i,j=1}^N\frac{|v_i(s) - v_j(s)|^2}{(|x_i(s) - x_j(s)|-\delta)^2} ds\cr
&\quad \leq \lt|\frac{1}{N^2}\sum_{i,j = 1}^N\log \lt(|x_i(0) - x_j(0)|-\delta\rt)\rt| + \frac{T_0}{2} + \frac{1}{2N}\sum_{i=1}^N |v_i(0)|^2 \quad \mbox{for} \quad t \in [0,T_0).
\end{aligned}$$
This yields \eqref{estim1d} in case of $\alpha=2$.

(ii) $\alpha > 2$: It is a straightforward to get
$$\begin{aligned}
\frac{d \ml^\beta(t)}{dt} &= -\beta\frac{1}{N^2}\sum_{i,j=1}^N (|x_i(t) - x_j(t)|-\delta)^{-\beta - 1}\frac{(x_i(t) - x_j(t))\cdot (v_i(t) - v_j(t))}{|x_i(t)-x_j(t)|}\cr
&\leq C\frac{1}{N^2}\sum_{i,j = 1}^N (|x_i(t) - x_j(t)|-\delta)^{-\beta - 1}|v_i(t) - v_j(t)|\cr
&\leq C\frac{1}{N^2}\sum_{i,j = 1}^N\frac{1}{(|x_i(t) - x_j(t)|-\delta)^\beta} + C\frac{1}{N^2}\sum_{i,j = 1}^N\frac{|v_i(t) - v_j(t)|^2}{(|x_i(t) - x_j(t)|-\delta)^{\beta +2}}\cr
&= C\ml^\beta(t) + C\frac{1}{N^2}\sum_{i,j = 1}^N\frac{|v_i(t) - v_j(t)|^2}{(|x_i(t) - x_j(t)|-\delta)^{\beta +2}} \quad \mbox{for} \quad t \in [0,T_0),
\end{aligned}$$
where we used the Young's inequality. Then, by using Gronwall's inequality, we obtain
\[
\ml^\beta(t) \leq \ml^\beta(0)e^{Ct} + Ce^{Ct}\int_0^{t} \frac{1}{N^2}\sum_{i,j = 1}^N\frac{|v_i(s) - v_j(s)|^2}{(|x_i(s) - x_j(s)|-\delta)^{\beta +2}} ds \quad \mbox{for} \quad t \in [0,T_0).
\]
We choose $\alpha = \beta + 2$ together with \eqref{est_energy} to deduce
\[
\ml^{\alpha - 2}(t) \leq \ml^{\alpha - 2}(0)e^{Ct} + Ce^{Ct}\frac{1}{N}\sum_{i=1}^N |v_i(0)|^2 \quad \mbox{for} \quad t \in [0,T_0)
\]
and the proof of \eqref{estim2d} is finished.
Hence if $\alpha \geq 2$, then there is no collision between particles  for all time and $T_0=\infty$.
\end{proof}
\begin{rem}Similarly as before, we also have the flocking estimates for the system \eqref{dcspart}. Suppose that the initial configurations $(x_0,v_0)$ satisfy
\[
\min_{1 \leq i \neq j \leq N}|x_{i0} - x_{j0}| > \delta,
 \quad \|x_0 - x_c(0)\|_{\ell^\infty} > \delta, \quad \mbox{and} \quad \|v_0 - v_c(0)\|_{\ell^\infty} < \frac12 \int_{2\|x_0 - x_c(0)\|_{\ell^\infty}}^\infty \psi(s - \delta)\,ds,
\]
Then if $\alpha \geq 2$ there exist positive constants $c_1 > c_0 > \delta > 0$ such that
\[
\|x(t) - x_c(0)\|_{\ell^\infty} \in [c_0,c_1] \quad \mbox{and} \quad \|v(t) - v_c(0)\|_{\ell^\infty} \leq \|v_0 - v_c(0)\|_{\ell^\infty} e^{-\psi(2c_1) t},
\]
for $t \geq 0.$ 
\end{rem}

%
%
%
%

\section*{Acknowledgments}
\small{JAC was partially supported by the Royal Society via a Wolfson Research Merit Award. YPC was supported by the ERC-Starting grant HDSPCONTR ``High-Dimensional Sparse Optimal Control''.  JAC and YPC were partially supported by EPSRC grant EP/K008404/1. YPC is also supported by the Alexander Humboldt Foundation through the Humboldt Research Fellowship for Postdoctoral Researchers. JP was supported by the Polish NCN grant PRELUDIUM 2013/09/N/ST1/04113}

%
%
%
%
\bibliographystyle{abbrv}
\bibliography{ccmp}{}

\end{document}